\documentclass[12pt]{amsart}
 \usepackage{amssymb,amsmath,amsfonts,latexsym}
 \usepackage{bm}
 \usepackage{array,graphics,color}
 \theoremstyle{plain}
 \newtheorem{propn}{Proposition}[section]
 \newtheorem{thm}[propn]{Theorem}
 \newtheorem{lemma}[propn]{Lemma}
 \newtheorem{cor}[propn]{Corollary}
 \newtheorem*{thm*}{Theorem}
 \theoremstyle{definition}

 \theoremstyle{remark}

 \newcommand{\q}{\mathcal{Q}}
 \newcommand{\s}{\mathcal{S}}

\newcommand{\Nat}{\mathbb{N}}
\newcommand{\Comp}{\mathbb{C}}

 \newcommand{\D}{\mathbb{D}}
 
 \newcommand{\ot}{\otimes}

\newcommand{\clb}{\mathcal{B}}

\newcommand{\clq}{\mathcal{Q}}

\newcommand{\cls}{\mathcal{S}}

\newcommand{\clw}{\mathcal{W}}

\begin{document}

\title{Rudin's Submodules of $H^2(\mathbb{D}^2)$}

\author[Das] {B. Krishna Das}

\address{
(B. K. Das) Indian Statistical Institute \\ Statistics and
Mathematics Unit \\ 8th Mile, Mysore Road \\ Bangalore \\ 560059 \\
India}

\email{dasb@isibang.ac.in, bata436@gmail.com}

\author[Sarkar]{Jaydeb Sarkar}

\address{
 (J. Sarkar) Indian Statistical Institute \\ Statistics and
 Mathematics Unit \\ 8th Mile, Mysore Road \\ Bangalore \\ 560059 \\
 India}

 \email{jay@isibang.ac.in, jaydeb@gmail.com}


\begin{abstract}
Let $\{\alpha_n\}_{n\ge 0}$ be a sequence of scalars in the
open unit disc of $\mathbb{C}$, and let $\{l_n\}_{n\ge 0}$ be a sequence of
natural numbers satisfying $\sum_{n=0}^\infty (1 - l_n|\alpha_n|)
<\infty$. Then the joint $(M_{z_1}, M_{z_2})$ invariant subspace
\[\cls_{\Phi} = \vee_{n=0}^\infty \Big( z_1^n
\prod_{k=n}^\infty \left(\frac{-\bar{\alpha}_k}{|\alpha_k|}
\frac{z_2 - \alpha_k}{1 - \bar{\alpha}_k z_2}\right)^{l_k}
H^2(\mathbb{D}^2)\Big),\] is called a Rudin submodule. In this paper
we analyze the class of Rudin submodules and prove that \[
\text{dim} (\cls_{\Phi}\ominus (z_1 \cls_{\Phi}+ z_2\cls_{\Phi}))=
1+\#\{n\ge 0: \alpha_n=0\}<\infty.
\]In particular, this answer a question earlier raised by Douglas
and Yang (2000) ~\cite{DY}.
\end{abstract}

\maketitle

\section{Introduction}
Let $H^2(\D)$ denote the Hardy space over the unit disc $\D = \{z
\in \mathbb{C}: |z| < 1\}$. We also say that $H^2(\D)$ is the
\textit{Hardy module} over $\D$. The Hilbert space tensor product
$H^2(\D) \otimes H^2(\D)$ is called the \textit{Hardy module} over
$\mathbb{D}^2$ and is denoted by $H^2(\mathbb{D}^2)$. As is well
known, every vector in $H^2(\D^2)$ can be represented as square
summable power series over $\D^2$ and the multiplication operators
by the coordinate functions $(M_{z_1}, M_{z_2})$ are commuting and
doubly commuting isometries (see \cite{R}). We will
often identify $(M_{z_1}, M_{z_2})$ with $(M_z \otimes I_{H^2(\D)},
I_{H^2(\D)} \otimes M_z)$.

A closed subspace $\cls$ of $H^2(\D^2)$ (or $H^2(\D)$) is said to be
a \textit{submodule} if $\cls$ is invariant under $M_{z_1}$ and
$M_{z_2}$ (or $M_z$). Beurling's (cf. \cite{DP}) celebrated result
states that a closed subspace $\cls \subseteq H^2(\D)$ is a
submodule of $H^2(\D)$ if and only if $\cls = \theta H^2(\D)$ for
some bounded inner function $\theta \in H^\infty(\D)$. This is a
fundamental result which has far-reaching consequences. For
instance, it readily follows that $\cls = \theta H^2(\D)$ admits the
wandering subspace $\cls \ominus z \cls = \mathbb{C} \theta$. In
particular, $\cls \ominus z \cls$ is a generating set of $\cls$. The
same conclusion also holds when $\cls$ is one of the followings:
a submodule of the Bergman
module~\cite{ARS}, a doubly commuting submodule of $H^2(\D^n)$ ~\cite{SSW}
and a doubly commuting submodule of the Bergman space or the Dirichlet space 
over polydisc~\cite{CDSS}. This motivates us to look into the ``wandering
subspace'' $\clw_{\cls} := \cls \ominus (z_1 \cls + z_2 \cls) =
(\cls \ominus z_1 \cls) \cap (\cls \ominus z_2 \cls)$, and leads us
to ask whether $\clw_{\cls}$ is a generating set of $\cls$ or not,
where $\cls$ is a submodule of $H^2(\D^2)$. In general, however,
this question has a negative answer. Rudin \cite{R} demonstrated a
negative answer to this question by constructing a submodule $\cls$
of $H^2(\D^2)$ for which $\cls \ominus (z_1 \cls + z_2 \cls)$ is a
finite dimensional subspace but not a generating set of $\cls$.

Our motivation in this paper is the following: (1) to study a
natural class of submodules, namely ``generalized Rudin
submodules'', and (2) to compute the wandering dimensions of
generalized Rudin submodules. In particular, we are interested in
understanding the wandering subspace of Rudin submodules of
$H^2(\D^2)$. Our results, restricted to the case of Rudin's
submodule answers a question raised by Douglas and Yang (see Corollary~\ref{rudin}).
Also these
results are one important step in our program to understand the idea
of constructing new submodules and quotient modules out of old ones.

Given an inner function $\varphi \in H^\infty(\D)$, for notational
simplicity, we set
\[
\s_{\varphi}:=\varphi H^2(\D),\quad \text{and}\quad \q_{\varphi}:=
H^2(\D)\ominus \s_{\varphi}.
\]

We now turn to formulate our definition of generalized Rudin
submodule. Let $\Psi = \{\psi_n\}_{n=0}^\infty \subseteq
H^\infty(\mathbb{D})$ be a sequence of increasing inner functions
and $\Phi = \{\varphi_n\}_{n=0}^\infty \subseteq
H^\infty(\mathbb{D})$ be a sequence of decreasing inner functions.
Then the \textit{generalized Rudin submodule} corresponding to the
inner sequence $\Psi$ and $\Phi$ is denoted by $\cls_{\Psi, \Phi}$,
and defined by \[\cls_{\Psi, \Phi} = \mathop{\bigvee}_{n=0}^\infty
\big(\cls_{\psi_n} \otimes \cls_{\varphi_n}\big).\] Now let
$\{\alpha_n\}_{n\ge 0}$ be a sequence of points in $\D$ and
$\{l_i\}_{n=0}^\infty \subseteq \Nat$ such that
$\sum(1-l_i|\alpha_i|)<\infty$, and $\psi_n = z^n$, and
$\varphi_n:=\prod_{i=n}^{\infty}b_{\alpha_i}^{l_i}$, $n\ge 0$. Then
$\cls_{\Psi, \Phi}$ will be denoted by $\cls_{\Phi}$:
\[\cls_{\Phi} = \mathop{\bigvee}_{n=0}^\infty \big(\cls_{z^n} \otimes \cls_{\varphi_n}\big).\]
Here for each non-zero $\alpha \in \D$, we denote by $b_{\alpha}$ the
Blaschke factor $b_{\alpha}(z) : =
 \frac{-\bar{\alpha}}{|\alpha|}
\frac{z-\alpha}{1-\bar{\alpha}z}$ and for $\alpha=0$ we set $b_{0}(z):=z$.

\noindent The sequence of Blaschke products as defined above is
called the \textit{Rudin sequence}, and the submodule
$\cls_{{\Phi}}$ as defined above is called the \textit{Rudin
submodule} corresponding to the Rudin sequence ${\Phi}$. These
submodules are also called inner sequence based invariant subspaces
of $H^2(\D^2)$, and was studied by M. Seto and R. Yang \cite{SY},
Seto \cite{S1, S2, S3} and Izuchi et al. \cite{ III2}.

The main result of this paper states that
\[
\text{dim} (\cls_{\Phi}\ominus (z_1\cls_{\Phi}+z_2\cls_{\Phi}))=
1+\#\{n\ge 0: \alpha_n=0\}<\infty.
\]

The remainder of the paper is organized as follows. Section 2
collects necessary notations and contains preparatory materials,
which are an essential tool in what follows. After this preparatory
section, which contain also new results, the main theorems are
proved in Section 3.

\section{Preparatory results}

We begin with the following representations of $\cls_{\Psi,\Phi}$
(cf.~\cite{III2}).

\begin{lemma}\label{representation}
Let $\cls_{\Psi,\Phi}$ be a generalized Rudin submodule and
$\varphi_{-1}: = 0$. Then
\begin{equation}\label{s}
\cls_{\Psi,\Phi}=\bigvee_{n=0}^{\infty}\cls_{\psi_n}\ot
\cls_{\varphi_n} = \bigoplus_{n=0}^\infty \s_{\psi_n}\ot
(\cls_{\varphi_{n}}\ominus\cls_{\varphi_{n-1}}).
\end{equation}
\end{lemma}

\begin{proof}
First note that for all $n\ge 1$,
\[
\bigoplus_{j=0}^n (\cls_{\varphi_{j}}\ominus\cls_{\varphi_{j-1}})
=\cls_{\varphi_{n}}.
 \]
Then the required representation of $\cls_{\Psi,\Phi}$ can be obtained
from the above identity and the fact that
$\cls_{\psi_n}\subseteq\cls_{\psi_{n-1}}, (n\ge 1)$.
\end{proof}

Keeping the equality $\cls_{\Psi,\Phi}\ominus
(z_1\cls_{\Psi,\Phi}+z_2\cls_{\Psi,\Phi})= (\cls_{\Psi,\Phi}\ominus
z_1\cls_{\Psi,\Phi})\cap (\cls_{\Psi,\Phi}\ominus
z_2\cls_{\Psi,\Phi})$ in mind we pass to describe the closed
subspace $(\cls_{\Psi,\Phi}\ominus z_1\cls_{\Psi,\Phi})$.

\begin{lemma}\label{zwandering}
Let $\cls_{\Psi,\Phi}$ be a generalized Rudin's submodule and
$\varphi_{-1}:=0$. Then
\[
(\cls_{\Psi,\Phi}\ominus
z_1\cls_{\Psi,\Phi})=\bigoplus_{n=0}^{\infty} \Comp \psi_n\otimes
(\cls_{\varphi_n}\ominus \cls_{\varphi_{n-1}}).
\]
\end{lemma}
\begin{proof}
By Lemma ~\ref{representation} it follows that
\[\begin{split}\cls_{\Psi,\Phi}\ominus z_1 \cls_{\Psi,\Phi} & =
\big( \bigoplus_{n=0}^\infty \s_{\psi_n}\ot
(\cls_{\varphi_{n}}\ominus\cls_{\varphi_{n-1}})\big) \ominus
z_1\big( \bigoplus_{n=0}^\infty \s_{\psi_n}\ot
(\cls_{\varphi_{n}}\ominus\cls_{\varphi_{n-1}})\big)\\& =
\bigoplus_{n=0}^\infty (\s_{\psi_n} \ominus z_1 \s_{\psi_n}) \ot
(\cls_{\varphi_{n}}\ominus\cls_{\varphi_{n-1}}).
\end{split}
\]
Thus the result follows from the fact that $\cls_{\theta} \ominus z
\cls_{\theta} = \mathbb{C} \theta$, for any inner $\theta \in
H^{\infty}(\D)$.
\end{proof}


Before proceeding further, we first observe that for $\alpha \in \D$ and $m\ge 1$,
$\{b_{\alpha}^j M_z^*
b_{\alpha}\}_{j=0}^{m-1}$ is an orthogonal basis of the quotient
module $\clq_{b_{\alpha}^m}$.

\begin{lemma}\label{key1}
Let $\theta_1, \theta_2$ be a pair of inner functions such that
$\theta_1 = b_{\alpha}^m \theta_2$ for some $\alpha \in \mathbb{D}$
and $m \geq 1$. Then
\[ \cls_{\theta_2}\ominus \cls_{\theta_1}= \theta_2 \clq_{b_{\alpha}^m} = \theta_2(H^2(\D)\ominus
b_{\alpha}^{m} H^2(\D)) = \bigoplus_{k=0}^{m-1} \Comp\theta_2
(b_{\alpha}^k M_z^*b_{\alpha}).
\]
In particular, $\cls_{\theta_2} \ominus \cls_{\theta_1}$ is an $m$-dimensional
subspace of $H^2(\D)$.
\end{lemma}
\begin{proof}
The proof follows from the fact that $M_{\theta_2}$ is an isometry
and $\{b_{\alpha}^j M_z^* b_{\alpha}\}_{j=0}^{m-1}$ is an orthogonal
basis of $\clq_{b_{\alpha}^m}$.
\end{proof}

To proceed with our discussion it is useful to compute the matrix
representation of the operator $P_{\cls_{\theta_2} \ominus \cls_{\theta_1}}
M_z^*|_{\cls_{\theta_2} \ominus \cls_{\theta_1}}$ with respect to
the orthogonal basis $\{\theta_2 b_{\alpha}^j M_z^*
b_{\alpha}\}_{j=0}^{m-1}$, which will be used in the proof of the
main result of this paper. To this end let $v_j = \theta_2
b_{\alpha}^j M_z^* b_{\alpha}$ for all $j = 0, \ldots, m-1$. Then
\[\langle (P_{\cls_{\theta_2} \ominus \cls_{\theta_1}} M_z^*) v_j,
v_i \rangle = \langle M_z^* b_{\alpha}^j M_z^* b_{\alpha},
b_{\alpha}^i M_z^* b_{\alpha} \rangle.\]
Note that
\[M_z^* b_{\alpha} =-\frac{\bar{\alpha}}{|\alpha|} (1 - |\alpha|^2) \mathbb{S}(\cdot,
\alpha),\]where $\mathbb{S}(\cdot, \alpha)$ is the Szeg\"{o} kernel
on $\D$ defined by $\mathbb{S}(\cdot, \alpha)(z) = (1 - \bar{\alpha}
z)^{-1}$, $z \in \D$. Consequently, for $i = j$, we have
\[\langle (P_{\cls_{\theta_2} \ominus \cls_{\theta_1}} M_z^*) v_j,
v_i \rangle = \langle M_z^{*2} b_{\alpha}, M_z^* b_{\alpha} \rangle
= \bar{\alpha} (1 - |\alpha|^2),\]
for $i>j$,
\[\langle
(P_{\cls_{\theta_2} \ominus \cls_{\theta_1}} M_z^*) v_j, v_i \rangle
= \langle M_z^* b_{\alpha}^j M_z^* b_{\alpha}, b_{\alpha}^i M_z^*
b_{\alpha} \rangle = \langle  M_z^{*2}
b_{\alpha},b_{\alpha}^{i-j} M_z^* b_{\alpha} \rangle = 0,\]
and for $j>i$,
\[\begin{split}\langle (P_{\cls_{\theta_2} \ominus \cls_{\theta_1}} M_z^*)
v_j, v_i \rangle & = \langle M_z^* b_{\alpha}^j M_z^* b_{\alpha},
b_{\alpha}^i M_z^* b_{\alpha} \rangle =\langle M_z^*
b_{\alpha}^{j-i} M_z^* b_{\alpha}, M_z^* b_{\alpha} \rangle \\& = (1
- |\alpha|^2)^2 \langle M_z^* b_{\alpha}^{j-i} \mathbb{S}(\cdot,
{\alpha}), \mathbb{S}(\cdot, {\alpha}) \rangle = (- \alpha)^{j - i -
1} (1 - |\alpha|^2)^2,\end{split}\]where the last equality follows
from
\[\begin{split}\langle b_{\alpha}^{j-i} \mathbb{S}(\cdot, {\alpha}),
z\mathbb{S}(\cdot, {\alpha}) \rangle & = \langle b_{\alpha}^{j-i}
\mathbb{S}(\cdot, {\alpha}), b_{\alpha} + \alpha \mathbb{S}(\cdot,
{\alpha}) \rangle = \langle b_{\alpha}^{j-i} \mathbb{S}(\cdot,
{\alpha}), b_{\alpha} \rangle + \bar{\alpha} \langle
b_{\alpha}^{j-i} \mathbb{S}(\cdot, {\alpha}),
 \mathbb{S}(\cdot, {\alpha}) \rangle\\ & = \langle b_{\alpha}^{j-i-1}
 \mathbb{S}(\cdot, {\alpha}), \mathbb{S}(\cdot, 0)
\rangle + \bar{\alpha} \big(b_{\alpha}^{j-i} \mathbb{S}(\cdot,
\alpha)\big)(\alpha) = \big(b_{\alpha}^{j-i-1} \mathbb{S}(\cdot,
{\alpha})\big)(0) + 0\\ & = (- \alpha)^{j - i - 1}.
\end{split}\] Therefore,
\begin{align*}
\langle (P_{\cls_{\theta_2} \ominus \cls_{\theta_1}}
M_z^*|_{{\cls_{\theta_2} \ominus \cls_{\theta_1}}}) v_j, v_i \rangle
&=\left\{\begin{array}{cl}
0 & \text{if } j< i \\
\bar{\alpha}(1-|\alpha|^2)& \text {if } j = i\\
(-\alpha)^{j-i-1}(1-|\alpha|^2)^2&\text{if } j\ge i+1.
\end{array}\right.
\end{align*}
Finally, since $M_{\theta_2b_{\alpha}^i} \in \clb(H^2(\D))$ is an isometry for any $0\le i\le m-1$,
we have
\[\|v_i\|=\|\theta_2 b_{\alpha}^i M_z^* b_{\alpha}\| = \|M_z^* b_{\alpha}\| = (1 -
|\alpha|^2) \|\mathbb{S}(\cdot, \alpha)\| = (1 -
|\alpha|^2)^{\frac{1}{2}}. \]

The computations above then show that the matrix representation of
$P_{\cls_{\theta_2} \ominus \cls_{\theta_1}} M_z^*|_{\cls_{\theta_2}
\ominus \cls_{\theta_1}}$ with respect to the orthonormal basis
$\{\frac{1}{\sqrt{1 - |\alpha|^2}} v_j\}_{j=0}^{m-1}$ is the upper
triangular matrix with diagonal entries $\bar{\alpha}$ and off diagonal entries $(-\alpha)^{j-i-1}(1-|\alpha|^2)$.

\section{Main Results}

\begin{thm}
Let $\cls_{\Phi}$ be a Rudin submodule of $H^2(\D^2)$. Then
\[
\text{dim} (\cls_{\Phi}\ominus (z_1 \cls_{\Phi} + z_2 \cls_{\Phi}))=
1+\#\{n\ge 0: \alpha_n=0\}<\infty.
\]
\end{thm}
\begin{proof}
Since $\{n\ge 0: \alpha_n=0\}$ is  a finite set, it is enough to
show that the equality holds. First, observe that
\[\cls_{\Phi}\ominus (z_1 \cls_{\Phi} + z_2 \cls_{\Phi}) = (\cls_{\Phi} \ominus z_1
\cls_{\Phi}) \cap (\ker P_{\cls_{\Phi}}M_{z_2}^*|_{\cls_{\Phi}})
=\ker P_{\cls_{\Phi}}M_{z_2}^*|_{\cls_{\Phi} \ominus z_1
\cls_{\Phi}}.\] Now Lemmas \ref{representation} and
\ref{zwandering}, with $\psi_n = z^n$, $n \geq 0$, and $\varphi_{-1}
= 0$ implies that
\[\cls_{\Phi} = \bigoplus_{n\ge 0} \big( z^n H^2(\D)
\otimes(\cls_{\varphi_n}\ominus \cls_{\varphi_{n-1}})\big),\]and
\[\quad \cls_{\Phi}\ominus z_1\cls_{\Phi}= \bigoplus_{n\ge 0}
\big(\Comp z^n\otimes(\cls_{\varphi_n}\ominus
\cls_{\varphi_{n-1}})\big).\] Then
\[\begin{split}P_{\cls_{\Phi}}M_{z_2}^*(\cls_{\Phi}\ominus z_1\cls_{\Phi}) &
= P_{\cls_{\Phi}} \big(\bigoplus_{n\ge 0}
\big(\Comp z^n\otimes M_z^* (\cls_{\varphi_n}\ominus
\cls_{\varphi_{n-1}})\big)\big)  \\ &  = \bigoplus_{n\ge 0}
\big(\Comp z^n \otimes P_{\cls_{\varphi_n}\ominus
\cls_{\varphi_{n-1}}} M_z^*(\cls_{\varphi_n}\ominus
\cls_{\varphi_{n-1}})\big),
\end{split}\]
where for the last equality we have used the fact that
$P_{\s_{\varphi_{n-1}}}M_{z}^*(\cls_{\varphi_n}\ominus
\cls_{\varphi_{n-1}})=\{0\}$. Therefore,
\[\ker  P_{\cls_{\Phi}}M_{z_2}^*|_{\cls_{\Phi}\ominus
z_1\cls_{\Phi}}= \big(\Comp\otimes \ker
P_{\cls_{\varphi_0}}M_{z}^*|_{\cls_{\varphi_0}}\big)
\bigoplus_{n=1}^{\infty} \big(\Comp z^n\otimes \ker
(P_{\cls_{\varphi_n}\ominus\cls_{\varphi_{n-1}}}M_{z}^*|_
{\cls_{\varphi_n}\ominus \cls_{\varphi_{n-1}}})\big).
\]For the first term on the right-hand side, we have $\ker
P_{\cls_{\varphi_0}}M_{z}^*|_{\cls_{\varphi_0}}=\Comp\varphi_0$,
that is, \[\mbox{dim~}(\ker
P_{\cls_{\varphi_0}}M_{z}^*|_{\cls_{\varphi_0}}) = 1.\] On the other
hand, the representing matrix of
$P_{\cls_{\varphi_n}\ominus\cls_{\varphi_{n-1}}}M_{z}^*|_
{\cls_{\varphi_n}\ominus \cls_{\varphi_{n-1}}}$ with respect to the
orthonormal basis $\{\frac{1}{\sqrt{1 -
|\alpha_{n-1}|^2}}\varphi_{n}b_{\alpha_{n-1}}^{k}M_{z}^*b_{\alpha_{n-1}}:
0\le k<l_n\}$, as discussed at the end of the previous section, is
an upper triangular matrix with $\bar{\alpha}_{n-1}$ on the
diagonal. This matrix is invertible if and only if $\alpha_{n-1}\neq
0$. In the case of $\alpha_{n-1}=0$, since the supper diagonal entries
are $1$, the rank of the representing matrix is $l_n-1$ and hence
the kernel is one-dimensional. This completes the proof.
\end{proof}

As an illustration of the above theorem, we consider an explicit example
of Rudin submodule. Let $\s_R$ denotes the submodule of $H^2(\D^2)$,
consisting those functions which have zero of order at least $n$ at
$(0,\alpha_n):= (0, 1-n^{-3})$. This submodule was introduced by W.
Rudin in the context of infinite rank submodules of $H^2(\D^2)$. It
is also well known that such $\s_R$ is a Rudin submodule (see
\cite{SY}, \cite{S2}, \cite{S3}), that is, $\cls_R = \cls_{\Phi}$
where
\[
\varphi_0=\prod_{i=1}^{\infty}b_{\alpha_i}^i,\quad\quad
\varphi_n=\frac{\varphi_{n-1}}
{\prod_{j=n}^{\infty}b_{\alpha_j}}\quad (n\ge 1).
\]
In this case, for all $n\ge 1$, $\s_{\varphi_n}\ominus\s_{\varphi_{n-1}}=\varphi_n(H^2(\D)\ominus
\prod_{j=n}^{\infty}b_{\alpha_j}H^2(\D))$. Set
\[e_k:=\big(\prod_{j=k+1}^{\infty}b_{\alpha_j}\big)M_z^*b_{\alpha_k}\quad (k\ge 1).\]
Then $\{\varphi_n e_k\}_{k=n}^{\infty}$ is an orthogonal basis for
$\s_{\varphi_n}\ominus\s_{\varphi_{n-1}}$, for all $n\ge 1$. A
similar calculation, as in the end of the previous section, shows
that the matrix representation of the operator
$P_{\cls_{\varphi_n}\ominus\cls_{\varphi_{n-1}}}M_{z}^*|_
{\cls_{\varphi_n}\ominus \cls_{\varphi_{n-1}}}$ with respect to
$\{\varphi_n e_k\}_{k=n}^{\infty}$ is again an upper triangular
infinite matrix with diagonal entries $\bar{\alpha_k}$ $(k\ge n)$.
Therefore $\ker
P_{\cls_{\varphi_n}\ominus\cls_{\varphi_{n-1}}}M_{z}^*|_
{\cls_{\varphi_n}\ominus \cls_{\varphi_{n-1}}}=\{0\}$ for all $n>1$.
Now for $n=1$, since $\alpha_1=0$ its kernel is one dimensional. By
the same argument, as in the proof of the above theorem, we obtain
the following result.
\begin{cor}
\label{rudin}
Let $\s_{R}$ be the Rudin submodule as above. Then $\s_R\ominus(z_1\s_R+z_2\s_R)$ is a two dimensional
subspace giben by
\[\s_R\ominus(z_1 \s_R + z_2 \s_R)=\Comp \prod_{i=1}^{\infty}b_{\alpha_i}^i\oplus
\Comp \prod_{i=2}^{\infty}b_{\alpha_i}^i.
\]
\end{cor}

We end this note with an intriguing question, raised by Nakazi
\cite{N}: Does there exist a submodule $\s$ of $H^2(\D^2)$ with
$\mbox{rank~} \s = 1$, such that $\s\ominus(z_1 \s+ z_2 \s)$ is not
a generating set for $\s$?
Although we are unable to determine such (counter-)example, however, we do
have the following special example: let $\{\alpha_n\}_{n=0}^\infty
\subseteq \D \setminus \{0\}$ be a sequence of distinct points
and $\Phi=\{\varphi_n\}_{n\ge 0}$, and
$\varphi_n:=\prod_{i=n}^{\infty}b_{\alpha_i}^{l_i}$. Then, by the main result 
of this paper,
$\s_{\Phi}\ominus (z_1 \s_{\Phi}+ z_2 \s_{\Phi})=\Comp\otimes \Comp
\varphi_0$ is a one dimensional, non-generating subspace of
$\s_{\Phi}$. In fact it follows from \cite{III2} that the rank of
$\s_{\Phi}$ is $2$.

\noindent\textbf{Acknowledgment:} The first author is
grateful to Indian Statistical Institute, Bangalore Centre for warm
hospitality.

\end{document}